\documentclass[oneside,a4paper,10pt]{amsart}

\usepackage{amssymb}
\usepackage{amsthm}
\usepackage{graphicx}
\usepackage[T1]{fontenc}
\usepackage{hyperref}
\usepackage{multirow}
\usepackage{array}

\usepackage{geometry}
\usepackage{amsmath}

\newtheorem{theorem}{Theorem}[section]
\newtheorem{lemma}[theorem]{Lemma}
\newtheorem{corollary}[theorem]{Corollary}

\theoremstyle{definition}
\newtheorem{definition}[theorem]{Definition}
\newtheorem{proposition}[theorem]{Proposition}

\theoremstyle{remark}
\newtheorem{remark}[theorem]{Remark}

\numberwithin{equation}{section}

\author{Bartosz Naskr\k{e}cki}
\title[Mordell-Weil ranks of families of elliptic curves]{Mordell-Weil ranks of families of elliptic curves associated to Pythagorean triples}

\begin{document}
\begin{abstract}
We study the family of elliptic curves $y^2=x(x-a^2)(x-b^2)$ parametrized by Pythago\-rean triples $(a,b,c)$. We prove that for a generic triple the lower bound of the rank of the Mordell-Weil group over $\mathbb{Q}$ is $1$, and for some explicitly given infinite family the rank is $2$. To each family we attach an elliptic surface fibered over the projective line. We show that the lower bounds for the rank are optimal, in the sense that for each generic fiber of such an elliptic surface its corresponding Mordell-Weil group over the function field $\mathbb{Q}(t)$ has rank $1$ or $2$, respectively. In order to prove this, we compute the characteristic polynomials of the Frobenius automorphisms acting on the second $\ell$-adic cohomology groups attached to elliptic surfaces of Kodaira dimensions $0$ and $1$.
\end{abstract}

\maketitle

\section{Introduction}
\medskip\noindent
Consider a triple of integers $a,b$ and $c$ that satisfy the Pythagorean equation
\[a^2+b^2=c^2.\]
We intend to study a family of elliptic curves $E_{(a,b,c)}$
\begin{equation}\label{eq:pythagorean_family}
y^2=x(x-a^2)(x-b^2)
\end{equation}
parametrized by such triples.
The family (\ref{eq:pythagorean_family}) is similar to another family of curves
\begin{equation}\label{eq:Frey_family}
y^2=x(x-a^2)(x+b^2)
\end{equation}
with $a^2+b^2=c^2$ which is a special case of the well-known Frey family. For low conductors there are many curves of high Mordell-Weil rank (up to rank $6$) in the family (\ref{eq:pythagorean_family}). This is, however, usually not the case for the family (\ref{eq:Frey_family}), since generically it is of rank $0$.

The family (\ref{eq:pythagorean_family}) is equivalent to the family of curves in the Legendre form
\[y^2=x(x-1)(x-\lambda)\]
with the parameter $\lambda$ limited to rational numbers of the form
\[\lambda=\left(\frac{2t}{t^2-1}\right)^2,\]
for $t$ rational, not equal to $0$ or $\pm 1$. The Mordell-Weil rank of the family (\ref{eq:pythagorean_family}) was considered for the first time in the paper \cite{INK_Pythagorean}, where it was proven  that the group $E_{(a,b,c)}(\mathbb{Q})$ of rational points contains a point $(c^2,abc)$ of infinite order.

\medskip\noindent
In order to state our results, we need some extra notation. Consider the set
\[\mathcal{T}=\{(a,b,c)\in\mathbb{Z}^{3}:a^2+b^2=c^2,\quad ab\neq 0\}\]
of triples of integers that satisfy the Pythagorean equation and define a smooth curve in the family (\ref{eq:pythagorean_family}). For any triple $(a,b,c)\in\mathcal{T}$, the rank of the Mordell-Weil group of rational points on $y^2=x(x-a^2)(x-b^2)$ is at least one by \cite[Lemma 6.8]{INK_Pythagorean}. We define an infinite subset $\mathcal{S}$ of $\mathcal{T}$. A triple $(a,b,c)$ belongs to $\mathcal{S}$ if and only if its coordinates can be written in the form
\begin{align*}
a&=P^2-Q^2,\\
b&=2PQ,\\
c&=P^2+Q^2,
\end{align*}
where the fraction $\frac{P}{Q}=\frac{2pq}{p^2+5q^2}$ for some $p,q\in\mathbb{Z}$.

\medskip\noindent
Our first main result is the following statement.
\begin{theorem}\label{Ranks_one}
For infinitely many $(a,b,c)\in\mathcal{S}$ the curve
\[y^2=x(x-a^2)(x-b^2)\]
has the Mordell-Weil group of rank at least two. There are two linearly independent points
\[Q_{1}=\left(\frac{1}{2} (a+b-c)^2,\frac{1}{2} (a+b) (a+b-c)^2\right),\]
\[Q_{2}=\left(\frac{1}{2} a (a-c),\frac{1}{2} a b \frac{1}{k^2}\left(p^4-25 q^4\right)\right),\]
where $k=GCD(2pq,p^2+5q^2)$ and $p$ and $q$ are as above.
\end{theorem}

\begin{remark}
The set $\mathcal{S}$ splits into countably many nonempty subsets $\{\mathcal{C}_{i}\}_{i=1}^{\infty}$ 
\[\mathcal{S}=\cup_{i=1}^{\infty} \mathcal{C}_{i}\] 
that correspond to isomorphism classes of curves over $\mathbb{Q}$, cf. Proposition \ref{proposition:bijection_classes}.  For all but finitely many $i$ the rank of the Mordell-Weil group of $\mathbb{Q}$-rational points of the curve $y^2=x(x-a^2)(x-b^2)$, where $a^2+b^2=c^2$ and $(a,b,c)\in\mathcal{C}_{i}$ is at least two.
\end{remark}

\noindent
Concerning the generic rank of family (\ref{eq:pythagorean_family}) we have the following result.
\begin{theorem}\label{theorem:rank_K3}
The group of $\mathbb{Q}(t)$-rational points on the curve
\begin{equation}\label{eq:Legendre_family}
y^2=x(x-1) \left(x- \left( \frac{2t}{t^2-1}\right)^2 \right)
\end{equation}
is of rank one.
\end{theorem}
\medskip\noindent
We prove Theorem \ref{theorem:rank_K3} as an application of the Shioda-Tate formula. In fact, a stronger result holds. The rank of the group of $\overline{\mathbb{Q}}(t)$-rational points of the curve from Theorem \ref{Ranks_one} is equal to $2$ but only a subgroup of rank one is defined over $\mathbb{Q}(t)$. Similar investigation of the generic rank of the family (\ref{eq:Frey_family}) shows that the rank of the Mordell-Weil group of the corresponding model over $\overline{\mathbb{Q}}(t)$ is equal to $0$.

\medskip\noindent
The result in Theorem \ref{Ranks_one} displays the generic rank but the corresponding geometric result is more involved.
\begin{theorem}\label{Ranks_two}
Let 
\[E:\quad y^2=x\left(x-\left(\left(\frac{2t}{t^2+5}\right)^2-1\right)^2\right)\left(x-4\left(\frac{2t}{t^2+5}\right)^2\right),\]
be the elliptic curve over $\overline{\mathbb{Q}}(t)$ which is obtained from the curve (\ref{eq:Legendre_family}) by a suitable change of parameter $t$ and a linear change of coordinates (cf. formula (\ref{eq:formula_proof})). The geometric Mordell-Weil group $E(\overline{\mathbb{Q}}(t))$ is isomorphic to $\mathbb{Z}^{3}\oplus\mathbb{Z}/2\mathbb{Z}\oplus\mathbb{Z}/4\mathbb{Z}$.
We put $u=\frac{2t}{t^2+5}$. The free part of the group $E(\overline{\mathbb{Q}}(t))$ is generated by the points
\begin{align*}
P_{1}&=(2 (1 + \sqrt{2}) (-1 + u)^2 u,2\sqrt{-1} (1 + \sqrt{2}) (-1 + (\sqrt{2} - u)^2) (-1 + u)^2 u),\\
P_{2}&=(2(u - 1)^2,2(-1 + u)^2 (-1 + 2u + u^2)),\\ 
P_{3}&=\left(1 - u^2, \frac{\left(-5 + t^2\right) u \left(-1 + u^2\right)}{5 + t^2}\right).
\end{align*}
The torsion subgroup of $E(\overline{\mathbb{Q}}(t))$ is generated by points
\begin{align*}
T_{1}&=(-4u^2,0)\\
T_{2}&=(2(-u + u^3),2\sqrt{-1}(u^2-1)u(-1 - 2u + u^2)).
\end{align*}
Moreover the group of $\mathbb{Q}(t)$-rational points on $E$ is generated by the points $P_{2}, P_{3}$ and $T_{1}$ and $2T_{2}$.
\end{theorem}
The proof of Theorem \ref{Ranks_two} requires more involved methods. Note that the geometric approach of Shioda, cf. \cite{Shioda_Mordell_Weil} implies only that the upper bound of the rank equals $6$. We base the proof of Theorem \ref{Ranks_two} on the approach of van Luijk in \cite{Luijk_Heron} and Kloosterman in \cite{Kloosterman_rank_15}. 

\medskip\noindent
To the best of our knowledge, the method of van Luijk and Kloosterman was used in the past exclusively for rational or K3 surfaces, cf. \cite{Schutt_Picard}, \cite{Luijk_Matrices}, \cite{Top_Zeeuw} and \cite{Elsenhans_Jahnel}. If an elliptic surface is of high geometric genus, then the method described below becomes very ineffective and it is computationally difficult to determine the zeta function of the surface. In our case, we perform calculations on elliptic sufaces which are rational or K3. In particular, we attach to the elliptic curve $E$ over $\overline{\mathbb{Q}}(t)$ an elliptic surface over $\mathbb{P}^{1}$, cf. \cite{Hulek}. We find its integral model $S$ as a scheme over $A$, where $A$ is a discrete valuation ring of a number field with a residue field isomorphic to $\mathbb{F}_{q}$. If the scheme $S\rightarrow A$ is smooth of relative dimension $2$ we obtain an elliptic surface $\tilde{S}=S_{\overline{\mathbb{F}_{q}}}$ over the field $\overline{\mathbb{F}_{q}}$. The action of the Frobenius automorphism on the second $\ell$-adic cohomology group $H^{2}_{et}(\tilde{S},\mathbb{Q}_{\ell})$, where $\ell\neq q$, gives rise to the characteristic polynomial of the automorphism. The computation of the characteristic polynomial involves point counting of $\mathbb{F}_{q^{r}}$-rational points on the surface $S_{\overline{\mathbb{F}_{q}}}$ up to some $r$. The Lefschetz fixed point formula allows us to compute the traces and the characteristic polynomials of the Frobenii. We apply \cite[Proposition 6.2]{Luijk_Heron} to estimate the number of eigenvalues of the form $p \zeta$, for some root of unity $\zeta$, which gives a sharp upper bound on the rank of the N\'{e}ron-Severi group $NS(S_{\overline{\mathbb{F}_{q}}})$. To conclude the computations, we apply the Shioda-Tate formula to obtain the rank of the group $E(\overline{\mathbb{Q}}(t))$. The rank of $E(\mathbb{Q}(t))$ is equal to $\mathop{rank}E(\overline{\mathbb{Q}}(t))-1$, because only one generator of the free part of $E(\overline{\mathbb{Q}}(t))$ is not defined over $\mathbb{Q}(t)$.  

\section{Notation and preliminaries}\label{section:proofs}
Let $S$ be the set of Pythagorean triples
\begin{equation}
S=\{(a,b,c)\in\mathbb{Z}^3:a^2+b^2=c^2\}.
\end{equation}

For each $s=(a,b,c)\in S$ we consider a curve over $\mathbb{Q}$
\begin{equation}
E_{s}: y^2=x(x-a^2)(x-b^2).
\end{equation}
When $ab\neq 0$ the equation defines a non-singular curve of genus one, hence an elliptic curve. The discriminant of the equation $E_{s}$ and its j-invariant are:
\begin{equation}
\Delta(s)=\Delta(a,b,c)=(16) (a - b)^2 (a + b)^2 b^4 a^4.
\end{equation}
\begin{equation}
j(s)=j(a,b,c)=256\cdot\frac{(a^4 - a^2 b^2 + b^4)^3}{b^4 a^4 (a - b)^2 (a + b)^2}.
\end{equation}

Observe that $\{s\in S: \Delta(s)\neq 0\}=\mathcal{T}$.

Let us now introduce the notion of equivalence of two tuples $s_{1},s_{2}\in \mathcal{T}$. We call two such tuples equivalent if two smooth curves $E_{s_{1}}$ and $E_{s_{2}}$ are equivalent via a linear change of coordinates defined over $\mathbb{Q}$, transforming one Weierstrass equation into another one, i.e. we assume that the curves are $\mathbb{Q}$-isomorphic.  We will write it as $s_{1}\sim s_{2}$. It is easy to check that $(a,b,c)\sim (A,B,C)$ if and only if there exists $u\in\mathbb{Q}^{\times}$ such that either $(a,b,c)=(\pm u A,\pm u B,\pm u C)$ or $(a,b,c)=(\pm u B,\pm u A,\pm u C)$. The relation $\sim$ is an equivalence relation. Hence, if $s_{1}$ and $s_{2}$ do not lie in the same equivalence class, the associated elliptic curves $E_{s_{1}}$ and $E_{s_{2}}$ are non-isomorphic over $\mathbb{Q}$. 

\medskip
\noindent
For any $s=(a,b,c)\in \mathcal{T}$ we introduce a new parameter $t=t(s)=\frac{b}{c-a}$. It is well-defined because a triple with $a=c$ cannot lie in $\mathcal{T}$. We have the following equalities:
\begin{equation}
\frac{t^2-1}{t^2+1}=\frac{a}{c},
\end{equation}
\begin{equation}
\frac{2t}{t^2+1}=\frac{b}{c}.
\end{equation}

Define an elliptic curve over $\mathbb{Q}(t)$:
\begin{equation}
E_{t}:y^2=x(x-(t^2-1)^2)(x-4t^2).
\end{equation}
A linear change of variables $x\mapsto x\frac{4}{(a-c)^2}, y\mapsto y\frac{8}{(c-a)^3}$ defines a $\mathbb{Q}$-isomorphism between elliptic curves $E_{(a,b,c)}$ and $E_{\frac{b}{c-a}}$ for any $(a,b,c)\in \mathcal{T}$.

The discriminant and j-invariant of the curve $E_{t}$ are
\[\Delta(t)=256 t^4 \left(-1+t^2\right)^4 \left(1-6 t^2+t^4\right)^2,\]
\[j(t)=\frac{16 \left(1-8 t^2+30 t^4-8 t^6+t^8\right)^3}{t^4 \left(-1+t^2\right)^4 \left(1-6 t^2+t^4\right)^2}.\]

The set $P=\{t\in\mathbb{Q}:\Delta(t)\neq 0\}=\mathbb{Q}\setminus\{0,\pm 1\}$ consists of all parameters for which $E_{t}$ is nonsingular. If $t\in P$, then
$E_{t}$ is $\mathbb{Q}$-isomorphic to a curve in Legendre form
\[y^2=x\left(x-1\right)\left(x-\left(\frac{2t}{t^2-1}\right)^2\right).\]
It is easy to check that two curves $E_{t}$ and $E_{t'}$ are $\mathbb{Q}$-isomorphic if and only if 
\[t'\in\left\{t,-t,\frac{1}{t},-\frac{1}{t},\frac{1+t}{1-t},\frac{1-t}{1+t},-\frac{1-t}{1+t},-\frac{1+t}{1-t}\right\}.\]
Rational functions in variable $t$, defined as above, form a group with composition. It is the dihedral group on $8$ elements generated by two mappings: $f(t)=-t$ and $g(t)=\frac{1+t}{1-t}$. An equivalence relation can be defined on the set $P$ where $t,t'\in P$ are in relation $t\sim t'$ if and only if the curves $E_{t}$ and $E_{t'}$ are $\mathbb{Q}$-isomorphic. Hence each equivalence class contains exactly 8 different elements.

\medskip\noindent
Consequently, we obtain the following 
\begin{proposition}\label{proposition:bijection_classes}
There is a a bijection of sets of equivalence classes
\begin{equation}
\mathcal{T}/\sim \rightarrow P/\sim
\end{equation}
given by $(a,b,c)\mapsto \frac{b}{c-a}$ on representatives. The inverse is given by \[\frac{p}{q}\mapsto (p^2-q^2,2pq,p^2+q^2).\]
\end{proposition}

\noindent
It follows from Proposition \ref{proposition:bijection_classes} that elements in $\mathcal{S}/\sim$ map bijectively to the elements in the set $\{u\in P: \exists_{t\in\mathbb{Q}}\quad u = \frac{2t}{5+t^2}\}/\sim$. The latter set is infinite, hence so is the former.

\section{Elliptic surfaces and Picard numbers}\label{section:elliptic_surf}
We start this section by recalling all neccessary theorems and definitions related to elliptic surfaces. We compute Picard numbers of several elliptic surfaces and deduce the generic rank of the Mordell-Weil group of elliptic curves related to family (\ref{eq:pythagorean_family}).

\begin{definition}
Let $k$ be an algebraically closed field. Let $C$ be a smooth, irreducible, projective curve over $k$. An \textit{elliptic surface over} $C$ is a smooth, irreducible, projective surface $S$ over $k$ with a relatively minimal elliptic fibration $f:S\rightarrow C$ with a singular fiber and a zero section.
\end{definition}

For an elliptic curve $E$ over the function field $k(C)$ of the curve $C$ we can associate an elliptic surface $f:\mathcal{E}\rightarrow C$ with generic fiber $E$. It follows from the work of Kodaira and N\'{e}ron that $f$ always exists and is unique. Further we refer to this elliptic surface as Kodaira-N\'{e}ron model of an elliptic curve $E$ over $k(C)$.

\medskip\noindent
Below we define three different elliptic surfaces, starting from three distinguished elliptic curves over the function field of $\mathbb{P}^{1}$ over $\overline{\mathbb{Q}}$. Let $\mathcal{E}_{1}\rightarrow\mathbb{P}^{1}$ be an elliptic surface over $\mathbb{P}^{1}$ associated to 
\[y^2=x(x-(t-1)^2)(x-4t).\]
Let $\mathcal{E}_{2}\rightarrow\mathbb{P}^{1}$ be an elliptic surface over $\mathbb{P}^{1}$ associated to 
\[y^2=x(x-(t^2-1)^2)(x-4t^2).\]
Finally, let $\mathcal{E}_{3}\rightarrow\mathbb{P}^{1}$ be an elliptic surface over $\mathbb{P}^{1}$ associated to 
\[y^2=x(x-(u^2-1)^2)(x-4u^2),\quad u=\frac{2t}{5+t^2}.\]

\medskip\noindent
For any smooth, projective, geometrically integral variety $V$ over a field $K$ we denote by $NS(V_{\overline{K}})$ the N\'{e}ron-Severi group, i.e. the group of divisors on $V$ modulo algebraic equivalence.

\begin{theorem}[\cite{Shioda_Mordell_Weil}, Corollary 2.2]
Let $S\rightarrow C$ be an elliptic surface. The N\'{e}ron-Severi group $NS(S)$ is finitely generated and torsion-free.
\end{theorem}

\begin{definition}
Let $S\rightarrow C$ be an elliptic surface. The \textit{Picard number} $\rho(S)$ of the surface $S$ is the rank of the N\'{e}ron-Severi group $NS(S)$.
\end{definition}

We recall the classical the Shioda-Tate formula.
\begin{theorem}[\cite{Shioda_Mordell_Weil}, Corollary 5.3]\label{theorem:Shioda_Tate}
Let $S\rightarrow C$ be an elliptic surface. Let $R\subset C$ be the set of points under singular fibers. For each $v\in R$ let $m_{v}$ denote the number of components of the singular fiber above $v$. Let $E$ denote the generic fiber of $S$ and $K$ be the function field of $C$. Let $\rho(S)$ denote the rank of N\'{e}ron-Severi group of $S$. We have the following identity
\[\rho(S)=2+\sum_{v\in R}(m_{v}-1)+\mathop{rank}(E(K)).\]
\end{theorem}

\begin{lemma}\label{lemma:torsion_types}
Let $E$ be an elliptic curve over $\overline{\mathbb{Q}}(t)$. Let $\Sigma\subset \mathbb{P}^{1}(\overline{\mathbb{Q}})$ be the set of points of bad reduction of E. Let $F_{v}$ denote the fibre at $v\in\Sigma$. We denote by $G(F_{v})$ the group generated by simple components of $F_{v}$. There exists an injective homomorphism
\[\phi:E\left(\overline{\mathbb{Q}}(t)\right)_{\mathop{tors}}\rightarrow \prod_{v\in \Sigma}G(F_{v}).\]
If $F_{v}$ is of multiplicative type $I_{n}$ in Kodaira notation (cf. \cite[Theorem IV.8.2]{Silverman_book}), the corresponding group is $\mathbb{Z}/n\mathbb{Z}$. If $F_{v}$ is of additive type $I_{2n}^{*}$, the group is $(\mathbb{Z}/2\mathbb{Z})^{2}$.
\end{lemma}
\begin{proof}
The map $\phi$ is a group homomorphism by \cite[Theorem IV.9.2]{Silverman_book}. It is injective by \cite[Corollary 7.5]{Shioda_Schutt}.
\end{proof}

\medskip\noindent
A multiplicative fiber of type $I_{n}$ has exactly $n$ components. An additive fiber of type $I_{2n}^{*}$ has $5+2n$ components.

\medskip\noindent
We gather the information about surfaces in Table \ref{table:singular_E1}, Table \ref{table:singular_E2} and Table \ref{table:singular_E3}.
We apply the Shioda-Tate formula to surfaces $\mathcal{E}_{1},\mathcal{E}_{2}$ and $\mathcal{E}_{3}$. 

\begin{lemma}\label{lemma:elliptic_types}

\noindent
\begin{enumerate}
\item[(1)] Surface $\mathcal{E}_{1}$ is of Kodaira dimension $-\infty$.
\item[(2)] Surface $\mathcal{E}_{2}$ is of Kodaira dimension $0$. 
\item[(3)] Surface $\mathcal{E}_{3}$ is of Kodaira dimension $1$.
\end{enumerate}
\end{lemma}
\begin{proof}
The Euler-Poincar\'{e} characteristic $e(S)$ of an elliptic surface $S\rightarrow C$ (over base field of characteristic different from $2$ and $3$) equals
\[e(S)=\sum_{v\in\Sigma}e(F_{v}),\]
where $\Sigma$ is the set of points over which there are singular fibers. The local Euler number $e(F_{v})$ is equal to the number of components $m_{v}$ if the fiber has multiplicative reduction, or to $m_{v}+1$ if the reduction is additive, cf. \cite[Proposition 5.1.6]{Dolgachev}. An easy computation with the Tate algorithm (cf. Table \ref{table:singular_E1}, Table \ref{table:singular_E2} and Table \ref{table:singular_E3}) shows that $e(\mathcal{E}_{1})=12$, $e(\mathcal{E}_{2})=24$ and $e(\mathcal{E}_{3})=48$. From \cite[Corollary V.12.3]{Hulek} it follows that the Kodaira dimensions are: $\kappa(\mathcal{E}_{1})=-\infty$, $\kappa(\mathcal{E}_{2})=0$ and $\kappa(\mathcal{E}_{3})=1$, respectively.
\end{proof}

\begin{table}[h]
\begin{tabular}{ccc}
Place & Type of singular fiber & Automorphism group \\
$t=1$ & $I_{4}$ & $\mathbb{Z}/4\mathbb{Z}$  \\
$t=0$ & $I_{2}$ & $\mathbb{Z}/2\mathbb{Z}$ \\
roots of $1 - 6 t + t^2=0$ & $I_{2}$ & $\mathbb{Z}/2\mathbb{Z}$ \\
$t=\infty$ & $I_{2}$ & $\mathbb{Z}/2\mathbb{Z}$\\
\end{tabular}
\caption{Singular fibers, $E_{1}:y^2=x(x-(t-1)^2)(x-4t)$ }\label{table:singular_E1}
\end{table}

\begin{table}[h]
\begin{tabular}{ccc}
place & Type of singular fiber & Automorphism group \\
$t=1$ & $I_{4}$ & $\mathbb{Z}/4\mathbb{Z}$ \\
$t=0$ & $I_{4}$ & $\mathbb{Z}/4\mathbb{Z}$ \\
$t=-1$ & $I_{4}$ & $\mathbb{Z}/4\mathbb{Z}$ \\
roots of $-1 - 2 t + t^2=0$ & $I_{2}$ & $\mathbb{Z}/2\mathbb{Z}$ \\
roots of $-1 + 2 t + t^2=0$ & $I_{2}$ & $\mathbb{Z}/2\mathbb{Z}$ \\
$t=\infty$ & $I_{4}$ & $\mathbb{Z}/4\mathbb{Z}$ \\
\end{tabular}
\caption{Singular fibers, $E_{2}:y^2=x(x-(t^2-1)^2)(x-4t^2)$ }\label{table:singular_E2}
\end{table}

\begin{table}[h]
\begin{tabular}{ccc}
Place & Type of singular fiber & Automorphism group \\
$t=0$ & $I_{4}$ & $\mathbb{Z}/4\mathbb{Z}$ \\
roots of $5 + t^2=0$ & $I_{4}$ & $\mathbb{Z}/4\mathbb{Z}$ \\
roots of $5 - 2 t + t^2=0$ & $I_{4}$ & $\mathbb{Z}/4\mathbb{Z}$ \\
roots of $5 + 2 t + t^2=0$ & $I_{4}$ & $\mathbb{Z}/4\mathbb{Z}$ \\
roots of $25 - 20 t + 6 t^2 - 4 t^3 + t^4=0$ & $I_{2}$ & $\mathbb{Z}/2\mathbb{Z}$ \\
roots of $25 + 20 t + 6 t^2 + 4 t^3 + t^4=0$ & $I_{2}$ & $\mathbb{Z}/2\mathbb{Z}$ \\
$t=\infty$ & $I_{4}$ & $\mathbb{Z}/4\mathbb{Z}$ \\
\end{tabular}
\caption{Singular fibers, $E_{3}:y^2=x(x-(u^2-1)^2)(x-4u^2)$, $u=\frac{2t}{5+t^2}$ }\label{table:singular_E3}
\end{table}

The information gathered in Table \ref{table:singular_E1}, Table \ref{table:singular_E2} and Table \ref{table:singular_E3} allows us to prove the following results.
\begin{lemma}
The generic fiber of $\mathcal{E}_{1}$ has rank $1$ over $\overline{\mathbb{Q}}(t)$.
\end{lemma}
\begin{proof}
The surface $\mathcal{E}_{1}$ is rational by Lemma \ref{lemma:elliptic_types}, hence $\rho(\mathcal{E}_{1})=10$. A section of $\mathcal{E}_{1}$ corresponds to a point on the generic fiber
\begin{equation}
P=(-4 t,4 \sqrt{-2} t (t+1)).
\end{equation}
As $2P$ and $4P$ are not zero, the point is non-torsion by Lemma \ref{lemma:torsion_types}. The group $E_{1}(\overline{\mathbb{Q}}(t))$ is at least of rank $1$.
By applying the Shioda-Tate formula we get
\[\rho(\mathcal{E}_{1})-2-\sum_{v\in R}(m_{v}-1)=\mathop{rank}(E_{1}(\overline{\mathbb{Q}}(t)))\]
\[\rho(\mathcal{E}_{1})-2-\sum_{v\in R}(m_{v}-1)=10-2-(4-1+4(2-1))=10-2-7=1.\]
Hence the rank equals $1$.
\end{proof}

\medskip
\medskip

\begin{lemma}\label{lemma:rank_e2}
The generic fiber of $\mathcal{E}_{2}$ has rank $2$ over $\overline{\mathbb{Q}}(t)$. 
\end{lemma}
\begin{proof}
The surface is $K3$ so $\rho(\mathcal{E}_{2})\leq 20$. Let us find two points of infinite order
\begin{eqnarray*}
P=(-4t^2,4\sqrt{-2}t^2(t^2+1)),\\
Q=(2(t - 1)^2,2(-1 + t)^2 (-1 + 2t + t^2)).
\end{eqnarray*}
We compute the height pairing $\langle P,P\rangle = 2$, $\langle Q,Q\rangle = 1$ and $\langle P,Q\rangle =0$. Application of the Shioda-Tate formula shows that 
\[
20\geq\rho(\mathcal{E}_{2})=2+(4(4-1)+4(2-1))+\mathop{rank}(E_{2}(\overline{\mathbb{Q}}(t)))=18+\mathop{rank}(E_{2}(\overline{\mathbb{Q}}(t)))
.\]
Hence the rank is equal to two.
\end{proof}

Application of the Shioda-Tate formula allows us to conclude that the Mordell-Weil group of $\overline{\mathbb{Q}}(t)$-rational points of the generic fiber of $\mathcal{E}_{3}$ has rank at most $6$. More precisely, $\rho(\mathcal{E}_{3})\leq 40$ (since $\chi(\mathcal{O}_{\mathcal{E}_{3}})=4$) and 
\[2+\sum_{v\in R}(m_{v}-1)=2+8(4-1)+8(2-1)=2+24+8=34.\]
There are only three sections of infinite order which are linearly independent.

\medskip\noindent
Let $X$ be any scheme over a finite field $\mathbb{F}_{q}$ of characteristic $p$. Let $\ell\neq p$ be a prime. Let us consider \'{e}tale $\ell$-adic cohomology groups $H_{\textrm{\'{e}t}}^{i}(X_{\overline{\mathbb{F}_{q}}},\mathbb{Q}_{l})=\underleftarrow{\lim} H_{\textrm{\'{e}t}}^{i}(X_{\overline{\mathbb{F}_{q}}},\mathbb{Z}/\ell^n)\otimes_{\mathbb{Z}_{\ell}}\mathbb{Q}_{\ell}$
and the groups with Tate twist $H_{\textrm{\'{e}t}}^{i}(X_{\overline{\mathbb{F}_{q}}},\mathbb{Q}_{l})(1)=H_{\textrm{\'{e}t}}^{i}(X_{\overline{\mathbb{F}_{q}}},\mathbb{Q}_{l})\otimes_{\mathbb{Z}_{\ell}}\left(\underleftarrow{lim} \mu_{\ell^{n}}\right)$, where $\mu_{\ell^n}\subset \overline{\mathbb{F}_{q}}$ is a group of $\ell^{n}$-th roots of unity. For simplicity we denote them by $H^{i}(X,\mathbb{Q}_{l})$ and $H^{i}(X,\mathbb{Q}_{l})(1)$. 

\begin{theorem}[\cite{Luijk_Heron}, Proposition 6.2]\label{theorem:Luijk_injective}
Let $A$ be a discrete valuation ring of a number field $L$ with the residue field $k\cong \mathbb{F}_{q}$. Let $S$ be an integral scheme with a morphism $S\rightarrow \mathop{Spec}A$ that is projective and smooth of relative dimension $2$. Let us assume that the surfaces $\overline{S}=S_{\overline{L}}$ and $\tilde{S}=S_{\overline{k}}$ are integral. Let $l\nmid q$ be a prime number. Then there are natural injective homomorphisms
\begin{equation}
NS(\overline{S})\otimes\mathbb{Q}_{l}\hookrightarrow NS(\tilde{S})\otimes\mathbb{Q}_{l}\hookrightarrow H^{2}_{\textrm{\'{e}t}}(\tilde{S},\mathbb{Q}_{l})(1)
\end{equation}
of finite dimensional inner product spaces over $\mathbb{Q}_{l}$. The first injection is induced by the natural injection $NS(\overline{S})\otimes\mathbb{Q}\hookrightarrow NS(\tilde{S})\otimes\mathbb{Q}$. The second injection respects the Galois action of $G(\overline{k}/k)$.
\end{theorem}

For any prime $p$ and any positive integer $r$ and a variety $X$ over $\mathbb{F}_{p^{r}}$ we denote by $F_{X}:X\rightarrow X$ the absolute Frobenius morphism which acts as the identity on points and as $f\mapsto f^p$ on the structure sheaf. Let $\Phi_{X}=(F_{X})^{r}$ and $\overline{X}=X_{\overline{\mathbb{F}}_{p^r}}$ and we denote by $\Phi_{X}\times 1$ the morphism which acts on $Z\times\mathop{Spec}\overline{\mathbb{F}}_{p^r}$. This induces an automorphism $\Phi_{X}^{*}$ of $H_{\textrm{\'{e}t}}^{i}(\overline{X},\mathbb{Q}_{l})$. 

\medskip
\begin{theorem}[\cite{Luijk_Heron}, Corollary 2.3]\label{theorem:Luijk_roots}
With notation as in the previous theorem, the ranks of $NS(\overline{S})$ and $NS(\tilde{S})$ are bounded from above by the number of eigenvalues of the linear map $\Phi_{\tilde{S}}^{*}$ for which $\lambda/q$ is a root of unity, counted with multiplicity.
\end{theorem}

To be able to use the above theorem effectively, we recall the Lefschetz trace formula, cf. \cite[VI, Theorem 12.3]{Milne_etale_book}.
\begin{theorem}\label{theorem:Lefschetz}
Let $X$ be a smooth projective variety over $\mathbb{F}_{q}$ of dimension $n$. For any prime $l\nmid q$ and any integer $m$, we have
\[\# X(\mathbb{F}_{q^m})=\sum_{i=0}^{n}(-1)^{i}Tr((\Phi_{X}^{*})^{m}\mid H^{i}(X_{\overline{\mathbb{F}_{q}}},\mathbb{Q}_{l})).\]
\end{theorem}

We explain the use of the Lefschetz trace formula in numerical computation of the characteristic polynomial of the Frobenius automorphism, which we apply in the proofs of Lemma \ref{lemma:rank_e1_prime} and Lemma \ref{lemma:rank_e1_bis} below.
We proceed with $X=S_{\overline{\mathbb{F}_{q}}}$, an elliptic surface fibered over $\mathbb{P}^{1}$. Note that $\dim H^{1}(X,\mathbb{Q}_{l})=\dim H^{3}(X,\mathbb{Q}_{l})$  by \cite[Corollary 2A10]{Kleiman_in_dix_expose} and $\dim H^{1}(X,\mathbb{Q}_{l})=0$ by \cite[Corollary 5.2.2]{Dolgachev}. Automorphism $\Phi_{X}^{*}$ acts on $H^{4}(X,\mathbb{Q}_{l})\cong\mathbb{Q}_{l}$ by multiplication by $q^2$. By the Lefschetz trace formula we obtain
\[\mathop{Tr}((\Phi_{X}^{*})^{m}\mid H^{2}(X,\mathbb{Q}_{l}))=\# X(\mathbb{F}_{q^m})-1-q^{2m}.\]
Let $V$ be a linear subspace of $H^{2}(X,\mathbb{Q}_{l})$ generated by components of singular fibers and sections. Let $W=H^{2}(X,\mathbb{Q}_{l})/V$. By the multiplicativity of the characteristic polynomial $\mathop{char}(\Phi_{X}^{*})$ we have that
\[\mathop{char}(\Phi_{X}^{*})=\mathop{char}(\Phi_{X}^{*}|V)\cdot \mathop{char}(\Phi_{X,W}^{*})\]
where the operator $\Phi_{X,W}^{*}:W\rightarrow W$ is induced by $\Phi_{X}^{*}$. Moreover,
\[\mathop{Tr}((\Phi_{X}^{*})^m)=\mathop{Tr}((\Phi_{X}^{*}|V)^m)+ \mathop{Tr}((\Phi_{X,W}^{*})^m)\]
for any $m$. For $T$ a linear operator acting on a finite dimensional vector space $U$, the characteristic polynomial $p(x)=\det(I\cdot x - T)$ can be computed if the traces $t_{n}=\mathop{Tr}(T^{n})$ are known for $0\leq n\leq \dim U$. In order to do that, we expand
\begin{equation}\label{equation:coefficients_traces}
p(x)=\frac{x^{\dim U}}{\mathop{exp}(\sum_{r=1}^{\infty}t_{r}\frac{x^{-r}}{r})}
\end{equation}
as a series of $\frac{1}{x}$ and truncate the series to the polynomial part. In numerical computations below, we put $T=\Phi_{X}^{*}$.

\section{Twisted elliptic surfaces}\label{subsection:twists}
Now we prove that the rank of the Mordell-Weil group over $\overline{\mathbb{Q}}(t)$ of the elliptic curve
\[y^2=x(x-(u^2-1)^2)(x-4u^2)\]
for $u=\frac{2t}{5+t^2}$ is equal to $3$. For this purpose the notion of a twist of an elliptic curve is needed. In this section we use the approach from the paper \cite{Kloosterman_rank_15} by Remke Kloosterman. We assume that the base curve $C$ of the elliptic fibration $\mathcal{E}\rightarrow C$ is defined over a field of characteristic not equal to $2$ or $3$.

\medskip\noindent
Let $C$ be a smooth curve over $k=\overline{k}$, and let $k(C)$ be the function field of $C$. Let $E$ be an elliptic curve over $k(C)$ given by the Weierstrass equation
\begin{equation}\label{eq:Weierstrass_equation}
E:y^2=x^3+Ax+B
\end{equation}
for $A,B\in k(C)$. Let us fix an element $u\in k(C)^{*}$. We consider the quadratic twist
\[E^{(u)}:uy^2=x^3+Ax+B\]
of the curve $E$ by the element $u$.

\begin{proposition}\label{theorem:ranks}
With the above notation the following equality holds
\[\mathop{rank}E(k(C))+\mathop{rank}E^{(u)}(k(C))=\mathop{rank}E(k(C)(\sqrt{u})).\]
\end{proposition}
\begin{proof}
See \cite[Exercise 10.16]{Silverman_arithmetic}.
\end{proof}

\begin{definition}\label{definition:twist}
Let $f:\mathcal{E}\rightarrow C$ be an elliptic surface. Fix two points $P,Q\in C(k)$. Let $E/k(C)$ be the generic fiber of $f$ given by the Weierstrass equation (\ref{eq:Weierstrass_equation}). An elliptic surface $g:\mathcal{E}'\rightarrow C$ \textit{is the twist of }$f$\textit{ by points} $P$\textit{ and }$Q$ if the generic fiber of $g$ is isomorphic over $k(C)$ to $E^{(u)}$, where $u\in k(C)^{*}$ and the function $u$ satisfies the following properties of the valuations:
\[\textrm{ord}_{P}(u)\equiv 1\textrm{ mod }2\]
and
\[\textrm{ord}_{Q}(u)\equiv 1\textrm{ mod }2.\]
In addition, we require that for all $R\neq P,Q$
\[\textrm{ord}_{R}(u)\equiv 0\textrm{ mod }2.\]
\end{definition}

\begin{remark}
For any pair of points $P$ and $Q$ we form a divisor $(P)-(Q)\in\textrm{Div}^{0}(C)$. The group of $k$-rational points of the Jacobian $\textrm{Jac}(C)$ of the curve $C$ equals $\textrm{Pic}^{0}(C)$. Since $k$ is algebraically closed, the group $\textrm{Pic}^{0}(C)$ is $2$-divisible. We find that $(P)-(Q)=2D+\textrm{div}(f)$ for a function $f\in k(C)^{*}$ and $D\in\textrm{Div}^{0}(C)$. We put $u:=f$. Let $u'$ be another function such that $\textrm{div}(u')\equiv(P)-(Q) \textrm{ mod }2\textrm{Div}^{0}(C)$. The twists $E^{(u)}$ and $E^{(u')}$ may not be isomorphic over $k(C)$. We have $\textrm{div}(\frac{u}{u'})=2T$ for some divisor $T\in\textrm{Div}^{0}(C)$. If the genus $g$ of $C$ is greater than $0$, then $T\in\textrm{Jac}(C)(k)$ is a $2$-torsion point. There are $2^{2g}$ distinct torsion points in $\textrm{Jac}(C)(k)$, hence there are $2^{2g}$ distinct twists by points $P,Q$, up to a $k(C)$-isomorphism. However, for $C=\mathbb{P}^{1}$, a pair of points $P,Q$ determines a twist uniquely.
\end{remark}

\begin{lemma}
Let $f:\mathcal{E}\rightarrow C$ be an elliptic surface and let $f^{(P,Q)}:\mathcal{E}^{(P,Q)}\rightarrow C$ be the twist by $P,Q\in C(k)$. There exists a double cover $\phi:C'\rightarrow C$ ramified at $P$ and $Q$ such that the relatively minimal nonsingular models of $\mathcal{E}\times_{C}C'\rightarrow C'$ and $\mathcal{E}^{(P,Q)}\times_{C}C'\rightarrow C'$ are isomorphic as fibered surfaces.
\end{lemma}
\begin{proof}
Let $E$ denote the generic fiber of the elliptic fibration $f$. Let $u$ be the function in $k(C)$ satisfying the conditions of Definition \ref{definition:twist} for the points $P$ and $Q$ in $C(k)$. The generic fiber of $f^{(P,Q)}$ is the twist $E^{(u)}$ of the curve $E$. 

\noindent
There exists a projective curve $C'$ and a surjective morphism $\phi:C'\rightarrow C$ such that $u\circ \phi=v^2$ for some element $v\in k(C')$. We denote by $e_{\phi}(R)$ the ramification index of the morphism $\phi$ at the point $R$ in the fiber above the point $\phi(R)\in C(k)$. By definition, the function $u$ has a divisor $\mathop{div}(u)=(P)+(Q)+2D$ for some $D\in\mathop{Div}(C)$. Hence
\begin{equation}\label{equation:div_properties}
\mathop{div}(u\circ\phi)=\phi^{*}(\mathop{div}u)=\sum_{R\in\phi^{-1}(P)}e_{\phi}(R)(R)+\sum_{R'\in\phi^{-1}(Q)}e_{\phi}(R')(R')+\phi^{*}D=2\mathop{div}v,
\end{equation}
where $\phi^{*}:\mathop{Div}(C)\rightarrow\mathop{Div}(C')$ denotes the induced map. The extension $k(C')$ of $k(C)$ is of degree $2$, so
\begin{equation}\label{equation:degree_properties}
2=\mathop{deg}\phi=\sum_{R\in\phi^{-1}(P)}e_{\phi}(R)=\sum_{R'\in\phi^{-1}(Q)}e_{\phi}(R').
\end{equation}
Identities (\ref{equation:div_properties}) and (\ref{equation:degree_properties}) imply that $\phi$ is ramified at $P$ and $Q$ and the preimages $\phi^{-1}(P)$ and $\phi^{-1}(Q)$ are singletons. 

\medskip\noindent
Let $S_{1}=\mathcal{E}\times_{C}C'$ and $S_{2}=\mathcal{E}^{(P,Q)}\times_{C}C'$ denote the surfaces obtained from $\mathcal{E}$ and $\mathcal{E}^{(P,Q)}$ by the base change $\phi:C'\rightarrow C$. The morphisms $f$ and $f^{(P,Q)}$ are projective, hence $S_{1}\rightarrow C'$ and $S_{2}\rightarrow C'$ are projective. As the base field $k$ is algebraically closed, all but finitely many fibers of $S_{1}\rightarrow C'$ and $S_{2}\rightarrow C'$ are nonsingular elliptic. Let $\tilde{S_{1}}$ denote a relatively minimal nonsingular model of $S_{1}$, respecting the fibration over $C'$. Similarly, let $\tilde{S_{2}}$ denote the relatively minimal nonsingular model of $S_{2}$. By a linear change of coordinates, the generic fibers $E$ and $E^{(u)}$ are isomorphic over $k(C')$. This implies that there is a birational map $\psi:\tilde{S_{1}}\dashrightarrow\tilde{S_{2}}$. Each such map is a composition of smooth blow-ups and blow-downs. The desingularizations $\tilde{S_{1}}$ and $\tilde{S_{2}}$ are isomorphic outside the singular fibers. The surfaces $\tilde{S_{1}}$ and $\tilde{S_{2}}$ are relatively minimal with respect to $C'$, so the fibers do not contain $(-1)$-curves. Hence, the map $\psi$ is a trivial composition, hence extends to an isomorphism.
\end{proof}

We introduce the following elliptic surfaces which will be used in the computation of ranks of the families associated to Pythagorean triples, cf. proofs of Theorems \ref{Ranks_one}, \ref{Ranks_two}.
\begin{definition}

\noindent
\begin{enumerate}
\item[(1)] Let $\mathcal{E}_{1}\rightarrow\mathbb{P}^{1}$ be the elliptic surface with the generic fiber
\[E_{1}:y^2=x(x-(t-1)^2)(x-4t).\]
\item[(2)] We denote by $\mathcal{E}_{1}'=\mathcal{E}_{1}^{(\frac{1}{5},\infty)}\rightarrow\mathbb{P}^{1}$ the twist of $\mathcal{E}_{1}$ by the points $\frac{1}{5}$ and $\infty$, which has the generic fiber
\[E_{1}':-(-1+5t)y^2=x(x-(t-1)^2)(x-4t).\]

\item[(3)] We denote by $\mathcal{E}_{1}''=(\mathcal{E}_{1}')^{(0,\infty)}\rightarrow\mathbb{P}^{1}$ the twist of $\mathcal{E}_{1}'$ by the points $0$ and $\infty$, which has the generic fiber
\[E_{1}'':-t(-1+5t)y^2=x(x-(t-1)^2)(x-4t).\]

\item[(4)] Let $\mathcal{E}_{2}\rightarrow \mathbb{P}^{1}$ be the elliptic surface with the generic fiber
\[E_{2}:y^2=x(x-(t^2-1)^2)(x-4t^2).\]
\item[(5)] We denote by $\mathcal{E}_{2}'=\mathcal{E}_{2}^{(\frac{-1}{\sqrt{5}},\frac{1}{\sqrt{5}})}\rightarrow \mathbb{P}^{1}$ the twist of $\mathcal{E}_{2}$ by the points $\frac{-1}{\sqrt{5}}$ and $\frac{1}{\sqrt{5}}$, which has the generic fiber
\[E_{2}':-(-1+5t^2)y^2=x(x-(t^2-1)^2)(x-4t^2).\]
\item[(6)] Let $\mathcal{E}_{3}\rightarrow\mathbb{P}^{1}$ be the elliptic surface which has the generic fiber
\[E_{3}: y^2=x(x-((\frac{2t}{5+t^2})^2-1)^2)(x-4(\frac{2t}{5+t^2})^2).\]
\end{enumerate}
\end{definition}

Proposition $\ref{theorem:ranks}$ implies the following statement.
\begin{corollary}\label{lemma:total_rank}
The following equalities hold
\[\mathop{rank}E_{3}(\overline{\mathbb{Q}}(t))=\mathop{rank}E_{2}(\overline{\mathbb{Q}}(t))+\mathop{rank}E_{2}'(\overline{\mathbb{Q}}(t)),\]
\[\mathop{rank}E_{2}'(\overline{\mathbb{Q}}(t))=\mathop{rank}E_{1}'(\overline{\mathbb{Q}}(t))+\mathop{rank}E_{1}''(\overline{\mathbb{Q}}(t)),\]
\end{corollary}

\begin{lemma}\label{lemma:rank_e1_prime}
The rank of $E_{1}'(\overline{\mathbb{Q}}(t))$ is equal to $0$.
\end{lemma}
\begin{proof}
First we perform the Tate algorithm to compute the types of singular fibers on our elliptic surface $\mathcal{E}_{1}^{'}\rightarrow\mathbb{P}^{1}$ associated with the curve $E_{1}^{'}$. A computation in MAGMA reveals that we have one fiber over the point $t=1$ of multiplicative type $I_{4}$, split over $\mathbb{Q}$. One singular fiber lies above $t=0$ and is non-split multiplicative of type $I_{2}$, nonetheless the equations are defined over $\mathbb{Q}$. We have a fiber over $t=\frac{1}{5}$, additive of type $I_{0}^{*}$ and again by the Tate algorithm and MAGMA the defining equations of the fiber have coefficents in $\mathbb{Q}$. The singular fiber over $t=\infty$ is additive of the type $I_{2}^{*}$ given by equations with coefficients in  $\mathbb{Q}$. Finally, we have two singular fibers of non-split multiplicative type $I_{2}$ above $t=3+2\sqrt{2}$ and $t=3-2\sqrt{2}$. The equations of the fibers are defined over $\mathbb{Q}(\sqrt{2})$ by the Tate algorithm. However, the surface $\mathcal{E}_{1}^{'}$ is defined over $\mathbb{Q}$, since we have started with the Weierstrass equation of the elliptic curve $E_{1}^{'}$ defined over $\mathbb{Q}$ and the singular locus defines an ideal where the generators have $\mathbb{Q}$-coefficients. In fact, $\mathcal{E}_{1}^{'}$ is defined over $\mathbb{Z}$. We check that the elliptic surface associated with $E_{1}^{'}$ over $\mathbb{F}_{17}$ has the types of singular fibers above the reductions of points $t=1,0,\infty, 3\pm 2\sqrt{2}$ the same as in characteristic zero. Put $\mathfrak{p}=(17)\in\mathop{Spec}\mathbb{Z}$ and $A=\mathbb{Z}_{(\mathfrak{p})}$. The surface $\mathcal{E}_{1}^{'}$ defines an integral scheme $S\rightarrow\mathop{Spec}A$ that is projective and smooth of relative dimension $2$. The smoothness comes from the fact that we have a good reduction at $17$. The residue field $k=A/\mathfrak{p}$ is equal to $\mathbb{F}_{17}$. Hence, the special fiber of $S\rightarrow\mathop{Spec}A$ is a surface defined over $\mathbb{F}_{17}$. It determines an elliptic surface $\tilde{S}=S_{\overline{\mathbb{F}_{17}}}\rightarrow\mathbb{P}^{1}$ which is the reduction of our elliptic surface $\mathcal{E}_{1}^{'}\rightarrow\mathbb{P}^{1}$. By Theorem \ref{theorem:Luijk_injective} we know that the rank of the N\'{e}ron-Severi group of $\mathcal{E}_{1}^{'}$ is bounded from above by the rank of the N\'{e}ron-Severi group of $\tilde{S}$. Components of the singular fibers and the zero section generate a rank $18$ subgroup in $NS(\tilde{S})$. The Euler-Poincar\'{e} characteristic $e(\mathcal{E}_{1}^{'})$ equals $24$ as follows by an argument based on the proof of Lemma \ref{lemma:elliptic_types}. Hence, the surface $\mathcal{E}_{1}^{'}$ is $K3$. Good reduction at prime $17$ implies that also $\tilde{S}$ is a $K3$ surface, so the subspace $\mathop{NS}(\tilde{S})\otimes\mathbb{Q}_{\ell}\hookrightarrow H^{2}_{\textrm{\'{e}t}}(\tilde{S},\mathbb{Q}_{\ell})(1)$ is at most of dimension $22$, because $\dim_{\mathbb{Q}_{\ell}}H^{2}_{\textrm{\'{e}t}}(\tilde{S},\mathbb{Q}_{\ell})(1)=\dim_{\mathbb{Q}_{\ell}}H^{2}_{\textrm{\'{e}t}}(\tilde{S},\mathbb{Q}_{\ell})=22$ by \cite[Theorem 4, Part III]{Mumford_Selected_papers}.
On the subspace $V$ generated by components of singular fibers and by zero section the Frobenius automorphism $\Phi_{\tilde{S}}^{*}$ acts by multiplication by $17$. It follows from the analysis of the singular fibers, i.e. by the Tate algorithm. The characteristic polynomial of the Frobenius automorphism $\Phi_{\tilde{S}}^{*}$ splits as follows
\[\mathop{char}(\Phi_{\tilde{S}}^{*})=\mathop{char}(\Phi_{\tilde{S}}^{*}\mid V)\cdot \mathop{char}(\Phi_{\tilde{S},H^{2}_{\textrm{\'{e}t}}/V}^{*}).\]
Then $\mathop{char}(\Phi_{\tilde{S}}^{*}\mid V)=\det(Id\cdot x-\Phi_{\tilde{S}}^{*}\mid V)=(x-17)^{18}$. For any natural $m$ an equality holds
\[\mathop{Tr}((\Phi_{\tilde{S}}^{*})^m)=\mathop{Tr}((\Phi_{\tilde{S}}^{*}\mid V)^m)+\mathop{Tr}((\Phi_{\tilde{S},H^{2}_{\textrm{\'{e}t}}/V}^{*})^m).\]
But we have $\mathop{Tr}((\Phi_{\tilde{S}}^{*}\mid V)^m)=18\cdot 17^{m}$ and $\mathop{Tr}((\Phi_{\tilde{S},H^{2}_{\textrm{\'{e}t}}/V}^{*})^m)=\#\tilde{S}(\mathbb{F}_{17^{m}})-1-17^{2m}$ by Lefschetz trace formula (cf. Theorem \ref{theorem:Lefschetz}). Combining those facts we obtain
\[\mathop{Tr}((\Phi_{\tilde{S},H^{2}_{\textrm{\'{e}t}}/V}^{*})^m)=\#\tilde{S}(\mathbb{F}_{17^{m}})-1-17^{2m}-18\cdot 17^{m}.\]
The characteristic polynomial $\mathop{char}(\Phi_{\tilde{S},H^{2}_{\textrm{\'{e}t}}/V}^{*})$ is of the form $x^4+c_{1}x^3+c_{2}x^2+c_{3}x+c_{4}$. We present explicit formulas  for $c_{m}$ in terms of $t_{m}=\mathop{Tr}((\Phi_{\tilde{S},H^{2}_{\textrm{\'{e}t}}/V}^{*})^m)$ (cf. equation (\ref{equation:coefficients_traces}))
\begin{align*}
c_{1}&=-t_{1}\\
c_{2}&= \frac{1}{2} (t_{1}^{2} - t_{2})\\
c_{3}&=\frac{1}{6}(-t_{1}^{3} + 3 t_{1} t_{2} - 2 t_{3})\\
c_{4}&=\frac{1}{24}(t_{1}^4 - 6 t_{1}^2 t_{2} + 3 t_{2}^2 + 8 t_{1} t_{3} - 6 t_{4})\\
\end{align*} 
We compute the number of $\mathbb{F}_{17^{m}}$-rational points on $\tilde{S}$ up to $m=4$.
\[\begin{array}{|c|c|c|c|c|}
\hline
m & 1 & 2 & 3 & 4\\
\hline
\#\tilde{S}(\mathbb{F}_{17^m}) & 604 & 88312 & 24227740 & 6977057176 \\
\hline
\end{array}\]
We obtain the characteristic polynomial $\mathop{char}(\Phi_{\tilde{S},H^{2}_{\textrm{\'{e}t}}/V}^{*})=x^4 - 8 x^3+ 238 x^2 - 2312 x +83521$. Suppose a root of this polynomial is $x=17\zeta$ for some root of unity $\zeta$. Then 
\[4913 (17 \zeta^4 - 8 \zeta^3 + 14 \zeta^2- 8 \zeta +17)=0.\]
But $\zeta$ is an algebraic integer and the polynomial  $17 x^4 - 8 x^3 + 14 x^2- 8 x +17$ is irreducible over $\mathbb{Q}$, hence its roots are not algebraic integers, which leads to a contradiction. Hence, the characteristic polynomial
$\mathop{char}(\Phi_{\tilde{S}}^{*})=(x-17)^{18} (x^4 - 8 x^3+ 238 x^2 - 2312 x +83521)$ has only 18 roots of the shape $17$ times a root of unity. By Theorem \ref{theorem:Luijk_roots} the rank of $\mathop{NS}(\tilde{S})$ is at most $18$. Then by Theorem \ref{theorem:Luijk_injective} the rank of $\mathop{NS}(\mathcal{E}_{1}^{'})$ is equal to $18$ since we have an explicit rank $18$ subgroup generated by singular fibers components and the zero section. By the Shioda-Tate formula the rank of $E_{1}^{'}(\overline{\mathbb{Q}}(t))$ equals zero. 
\end{proof}

\section{Computing ranks by reductions}
Let $p$ be a prime of good reduction for an elliptic surface $\mathcal{E}\rightarrow C$ defined over a number field $K$. Let $S$ be an integral model of $\mathcal{E}$ over $\mathcal{O}_{K}$ with special fiber defined over $\mathbb{F}_{p^{r}}$. We know by Theorem \ref{theorem:Luijk_injective} that 
\[NS(S_{\overline{\mathbb{Q}}})\otimes\mathbb{Q}\hookrightarrow NS(S_{\overline{\mathbb{F}}_{p^{r}}})\otimes\mathbb{Q}.\]
Assume for a moment that the map is an isomorphism. Then by classical results in lattice theory it follows that the determinants of the Gram matrices of the intersection pairings on $NS(S_{\overline{\mathbb{Q}}})$ and $NS(S_{\overline{\mathbb{F}}_{p^{r}}})$ differ by a square. In the sequel, we denote the determinant of the Gram matrix of a lattice $\Lambda$ by $\Delta(\Lambda)$. 

We will compute discriminants modulo squares using the Tate conjecture and the Artin-Tate conjecture for K3 surfaces which we recall for the reader's convenience.
\begin{theorem}\label{theorem:Tate_conjectures}
Let $Y$ be a K3 surface over $\mathbb{F}_{q}$. Let $\Phi_{Y}^{*}$ be the Frobenius automorphism acting on the cohomology group $H^{2}(Y,\mathbb{Q}_{l}))$, $l\nmid q$. The number of roots of the characteristic polynomial of $\Phi_{Y}^{*}$ of the form $q\zeta$, where $\zeta$ is a root of unity is equal to the Picard number $\rho(Y)=\mathop{rank}NS(Y_{\overline{\mathbb{F}}_{q}})$.
\end{theorem}

\begin{theorem}\label{theorem:Artin_Tate}
Let $Y$ be a K3 surface over $\mathbb{F}_{q}$. Let $\Phi_{Y}^{*}$ be the Frobenius automorphism and $P(T)=\det(1-T\Phi_{Y}^{*}|H^{2}(Y,\mathbb{Q}_{l}))$. Then
\[\lim_{s\rightarrow 1}\frac{P(q^{-s})}{(1-q^{1-s})^{\rho '(Y)}}=\frac{(-1)^{\rho '(Y)-1}\sharp Br(Y)\Delta(NS(Y_{\mathbb{F}_{q}}))}{q^{\alpha(Y)}(\sharp NS(Y_{\mathbb{F}_{q}})_{\textrm{tor}})^2},\]
where $\alpha(Y)=\chi(Y,\mathcal{O}_{Y})-1+\dim Pic^{0}(Y)$ and $Br(Y)$ is the Brauer group of $Y$. Moreover $\rho '(Y)=\mathop{rank} NS(Y_{\mathbb{F}_{q}})$. The group $NS(Y_{\mathbb{F}_{q}})$ is the subgroup of the N\'{e}ron-Severi group $NS(Y_{\overline{\mathbb{F}}_{q}})$ generated by $\mathbb{F}_{q}$-rational divisors. 
\end{theorem}
Tate conjectures for elliptic $K3$ surfaces are proven in \cite[Theorem 5.2]{Artin_SD_K3_Tate_Conjectures}. J. S. Milne proved that the Tate conjectures imply the Artin-Tate conjectures for characteristic different from $2$, cf. \cite[Theorem 6.1]{Milne_Tate_Conjecture}. Finally, in \cite[Theorem 0.4b]{Milne_Tate_Conjecture_2} the assumption on the characteristic was dropped.

\begin{proposition}[Proposition 4.7,\cite{Kloosterman_rank_15}]
Suppose $q$ is a prime power. Let $Y\rightarrow\mathbb{P}^{1}$ be an elliptic K3 surface, defined over $\mathbb{F}_{q}$. Assume that $q$ is a square and that $\rho(Y)=\rho '(Y)$. Then
\[\Delta(NS(Y_{\overline{\mathbb{F}}_{q}}))\equiv -\lim_{s\rightarrow 1}\frac{P(q^{-s})}{(1-q^{1-s})^{\rho(Y)}}\textrm{ mod }(\mathbb{Q}^{*})^{2}.\]
\end{proposition}

\begin{lemma}\label{lemma:rank_e1_bis}
The rank of $E_{1}''(\overline{\mathbb{Q}}(t))$ is equal to $1$.
\end{lemma}
\begin{proof}
It is easy to check that the point $Q=(1-t,1-t)$ lies in $E_{1}''(\overline{\mathbb{Q}}(t))$ and that it is a point of infinite order. The configuration of singular fibers is given in Table \ref{table:singular_E1bis}.
\begin{table}[h]
\begin{tabular}{ccc}
place & Type of singular fiber & Automorphism group \\
$t=1$ & $I_{4}$ & $\mathbb{Z}/4\mathbb{Z}$\\
$t=\infty$ & $I_{2}$ & $\mathbb{Z}/2\mathbb{Z}$\\
$t=0$ & $I_{2}^{*}$ & $(\mathbb{Z}/2\mathbb{Z})^2$\\
$t=\frac{1}{5}$ & $I_{0}^{*}$ & $(\mathbb{Z}/2\mathbb{Z})^2$\\
$t=3+\sqrt{2}$ & $I_{2}$ & $(\mathbb{Z}/2\mathbb{Z})$\\
$t=3-\sqrt{2}$ & $I_{2}$ & $(\mathbb{Z}/2\mathbb{Z})$\\
\end{tabular}
\caption{Singular fibers, $E_{1}'':-t(-1+5t)y^2=x(x-(t-1)^2)(x-4t)$ }\label{table:singular_E1bis}
\end{table}
It follows by Lemma \ref{lemma:torsion_types} that it is enough to check that $2Q$ and $4Q$ are non-zero. The Euler characteristic $e(\mathcal{E}_{1}'')=24$, which shows that $\mathcal{E}_{1}''$ is a $K3$ surface (cf. Table \ref{table:singular_E1bis}). Surface $\mathcal{E}_{1}^{''}$ is defined over $\mathbb{Z}$ (cf. proof of Lemma \ref{lemma:rank_e1_prime}). We have two primes of good reduction $11$ and $17$. Consider the reduction of $\mathcal{E}_{1}^{''}$ at $11$, which we denote by $S_{11}$. It is a K3 surface defined over $\mathbb{F}_{11}$. We have also a K3 surface obtained by the reduction at $17$. We denote it by $S_{17}$. Note that it is defined over $\mathbb{F}_{17}$. Since we are interested only in the surfaces defined over $\mathbb{F}_{11^2}$ and $\mathbb{F}_{17^2}$, we will denote by $S_{11}$ and $S_{17}$ the base change of original surfaces to $\mathbb{F}_{11^2}$ and $\mathbb{F}_{17^2}$, respectively.

By an argument similar to the proof of Lemma \ref{lemma:rank_e1_prime} we compute the characteristic polynomials of the  Frobenius automorphism acting on the second $\ell$-adic cohomology group for some auxiliary prime $\ell\neq 11, 17$. 

For $p=11$ using MAGMA we get 
\[\mathop{char}(\Phi_{S_{11}}^{*})=(x - 11^2)^{20} (x^2 - 158 x +14641).\]
Roots of the polynomial $x^2 - 158 x +14641$ are not of the form $11^2\zeta$, for some root of unity $\zeta$. The rank of $NS((S_{11})_{\overline{\mathbb{F}}_{11^2}})$ equals  $20$ by Tate conjectures for K3 surfaces (cf. Theorem \ref{theorem:Tate_conjectures})
For $p=17$ we get
\[\mathop{char}(\Phi_{S_{17}}^{*})=(x-17^2)^{20} ( x^2 + 94 x + 83521).\]
Roots of the polynomial $x^2 + 94 x + 83521$ are not of the form $17^2\zeta$, for some root of unity $\zeta$. Hence, the rank of $\mathop{NS}((S_{17})_{\overline{\mathbb{F}}_{17^2}})$ equals $20$ by Tate conjectures for K3 surfaces. The rank of $\mathop{NS}((\mathcal{E}_{1}^{''})_{\overline{\mathbb{Q}}})$ is always smaller or equal to the rank of the corresponding N\'{e}ron-Severi group after reduction (cf. Theorem \ref{theorem:Luijk_injective}). Assume for a moment that it is maximal possible, hence equal to $20$. This implies that the discriminants of lattices $\mathop{NS}((S_{11})_{\overline{\mathbb{F}}_{11^2}})$ and $\mathop{NS}((S_{17})_{\overline{\mathbb{F}}_{17^2}})$ should differ by a square.
We apply Theorem \ref{theorem:Artin_Tate} to compute the discriminant of the N\'{e}ron-Severi lattices $NS(\tilde{S}_{\overline{\mathbb{F}}_{17^2}})$ and $NS((S_{11})_{\overline{\mathbb{F}}_{11^2}})$. They are not equal modulo squares, more precisely we have
\[\Delta (NS((S_{11})_{\overline{\mathbb{F}}_{11^2}}))\equiv -3\cdot 7\textrm{ mod }(\mathbb{Q}^{*})^2\]
and
\[\Delta (NS((S_{17})_{\overline{\mathbb{F}}_{17^2}}))\equiv -2\cdot 3\cdot 7\textrm{ mod }(\mathbb{Q}^{*})^2.\]
So the rank of $\mathop{NS}((\mathcal{E}_{1}^{''})_{\overline{\mathbb{Q}}})$ is less or equal to $19$. Note that the trivial sublattice generated by components of singular fibers and the zero section is of rank $18$. We also have the point $Q$ of infinite order, so $19\leq\mathop{rank}\mathop{NS}((\mathcal{E}_{1}^{''})_{\overline{\mathbb{Q}}})$. But the upper bound is also $19$, hence the rank equals $19$. Now an application of the Shioda-Tate formula reveals that the rank of $E_{1}''(\overline{Q}(t))$ is equal to $1$. 
\end{proof}

\begin{corollary}\label{corollary:rank}
The rank of $E_{3}(\overline{\mathbb{Q}}(t))$ is equal to $3$.
\end{corollary}
\begin{proof}
We apply Lemma $\ref{lemma:total_rank}$ to ranks obtained in Lemma $\ref{lemma:rank_e1_prime}$, Lemma $\ref{lemma:rank_e1_bis}$ and Lemma $\ref{lemma:rank_e2}$. 
\end{proof}

\begin{remark}
One could give a more direct proof of Corollary \ref{corollary:rank} using brute force and more powerful numerical computations. The statement of Corollary \ref{corollary:rank} is equivalent to $\rho((\mathcal{E}_{3})_{\overline{\mathbb{Q}}})=37$ by the Shioda-Tate formula (cf. Table \ref{table:singular_E3} to compute the number of components in singular fibers). Suppose to the contrary that $\rho((\mathcal{E}_{3})_{\overline{\mathbb{Q}}})\geq 38$. This lower bound holds for the N\'{e}ron-Severi group of the reduced elliptic surface at primes of good reduction. Suppose we have two such primes $p_{1}$ and $p_{2}$. Prime $17$ is a good candidate, with the characteristic polynomial of the Frobenius automorphism (acting on the second cohomology group) equal to 
\[(t+17)^{8}(t-17)^{30}(289 - 22t + t^2)(289 - 2t + t^2)(83521 - 2312t + 238t^2 - 8t^3 + t^4).\]
To compute the degree $8$ factor we need to work with surfaces with points in the field $\mathbb{F}_{17^8}$ or $\mathbb{F}_{17^4}$ which follows by the Poincar\'{e} duality. None of the roots of 
\[(289 - 22t + t^2)(289 - 2t + t^2)(83521 - 2312t + 238t^2 - 8t^3 + t^4)\]
are of the shape $17\zeta$ for $\zeta$ a root of unity. Note that the Tate conjecture holds automatically for such a prime. By the results of J. S. Milne, cf. \cite[Theorem 6.1]{Milne_Tate_Conjecture} the Artin-Tate conjecture holds as well. Put $p_{1}=17$ and assume we have another such prime $p_{2}$. This means that we can compare the discriminants of the lattices modulo squares and arrive at a contradiction, which proves that $\rho((\mathcal{E}_{3})_{\overline{\mathbb{Q}}})=37$. Using the method of twists and further computation other good primes can be found, namely $73$ and $97$ but no other up to 140. However, a direct computation of the points on surfaces over $\mathbb{F}_{73^{4}}$ or $\mathbb{F}_{97^{4}}$ is beyond the range of our computational resources.
\end{remark}

\section{Proofs of main results}

\begin{lemma}\label{lemma:torsion}
The torsion subgroup of $E_{3}(\overline{\mathbb{Q}}(t))$ is isomorphic to $\mathbb{Z}/2\mathbb{Z}\oplus\mathbb{Z}/4\mathbb{Z}$. It is generated by points
\begin{align*}
T_{1}&=(-4u^2,0)\\
T_{2}&=(2(-u + u^3),2\sqrt{-1}(u^2-1)u(-1 - 2u + u^2)),
\end{align*}
where $u=\frac{2t}{5+t^2}$.
\end{lemma}
\begin{proof}
Let $K=\overline{\mathbb{Q}}(t)$. The elliptic surface associated to $E_{3}$ has singular fibers of types $I_{2}$ and $I_{4}$ (cf. Table \ref{table:singular_E3}), hence
\[E_{3}(K)_{\mathop{tors}}\hookrightarrow (\mathbb{Z}/2\mathbb{Z})^{a}\oplus(\mathbb{Z}/4\mathbb{Z})^{b}\]
for some natural numbers $a$ and $b$ by Lemma \ref{lemma:torsion_types}. The $2$-torsion subgroup is generated by $T_{1}$ and $(0,0)$. We will check that $T_{1}\notin 2E_{3}(K)$ and that $T_{1}+(0,0)\notin 2E_{3}(K)$, but $(0,0)\in 2E_{3}(K)$.

Let $P=(x,y)$ be any point in $E_{3}(K)$. Then the $x$-coordinate of $2P$ is equal to
\[x(2P)=\frac{\left(4 u^2-8 u^4+4 u^6-x^2\right)^2}{4 \left(4 u^2-x\right) \left(1-2 u^2+u^4-x\right) x},\]
where $u=\frac{2t}{5+t^2}$. If $T_{1}$ were in $2E_{3}(K)$, then
\[x(2P)=4u^2\] and in consequence the equation
\[16 t^2 (25 + 6 t^2 + t^4)^2 - 32 t^2 (5 + t^2)^4 x + (5 + t^2)^6 x^2= 0\]
would have a solution $x\in K$. The discriminant of the above quadratic polynomial is equal to
\[-64 t^2 (5 + t^2)^6 (625 - 100 t^2 - 74 t^4 - 4 t^6 + t^8)\]
and it is not a square in $K$, hence we get a contradiction. Similarly one can show that $T_{1}+(0,0)= ((u^2-1)^2,0)$ is not in $2E_{3}(K)$. Finally, it is easy to check that $2T_{2}=(0,0)$. The claim follows from that.
\end{proof}

\begin{lemma}\label{lemma:mordell_weil_generators}
The group $E_{3}(\overline{\mathbb{Q}}(t))/E_{3}(\overline{\mathbb{Q}}(t))_{\textrm{tors}}$ is free abelian of rank $3$. It is generated by the following points
\begin{align*}
P_{1}&=(2 (1 + \sqrt{2}) (-1 + u)^2 u,2\sqrt{-1} (1 + \sqrt{2}) (-1 + (\sqrt{2} - u)^2) (-1 + u)^2 u),\\
P_{2}&=(2(u - 1)^2,2(-1 + u)^2 (-1 + 2u + u^2)),\\
P_{3}&=(1 - u^2, \frac{(-5 + t^2) u (-1 + u^2)}{5 + t^2}),
\end{align*}
where $u=\frac{2t}{5+t^2}$.
\end{lemma}
\begin{proof}
We follow the argument in the proof of \cite[Proposition 4.2]{Luijk_Matrices}. We put $K=\overline{\mathbb{Q}}(t)$. Let $(E_{3}(K)/E_{3}(K)_{\mathop{tors}},\langle\cdot,\cdot\rangle_{E_{3}})$ denote the Mordell-Weil lattice with the height pairing $\langle\cdot,\cdot\rangle_{E_{3}}$. From the type of singular fibers, i.e. $I_{2}$ and $I_{4}$, cf. Table \ref{table:singular_E3}, we know that for each $P,Q\in E_{3}(K)/E_{3}(K)_{\mathop{tors}}$ we have $\langle P,Q\rangle_{E_{3}}\in \frac{1}{4}\mathbb{Z}$. 

Consider the lattice $\Lambda=E_{3}(K)/E_{3}(K)_{\mathop{tors}}$ with the pairing $\langle\cdot,\cdot\rangle=4\langle\cdot,\cdot\rangle_{E_{3}}$. Let $\Lambda '$ be generated by $P_{1}$, $P_{2}$ and $P_{3}$. It is a sublattice of $\Lambda$ of a finite index $n=[\Lambda:\Lambda ']$. In the lattice $\Lambda$ we have 
$\langle P_{i}, P_{i}\rangle = 4 i$ for $i=1,2,3$ and $\langle P_{i},P_{j}\rangle = 0$ for $i\neq j$. Hence, the following equality holds for discriminants of lattices $\Lambda$ and $\Lambda '$ with the pairing $\langle\cdot,\cdot\rangle$
\[6\cdot 4^3 = \Delta(\Lambda ') = n^2 \Delta(\Lambda).\]
Therefore, $n$ divides $8$. We want to show that $n=1$. Consider the $2$-descent homomorphism
\[\psi:E_{3}(K)/2E_{3}(K)\hookrightarrow K^{*}/(K^{*})^{2}\times K^{*}/(K^{*})^{2}.\]
The homomorphism $\psi$ is defined for points $(x,y)$ in $E_{3}(K)\setminus E_{3}(K)[2]$ by the formula \[\psi(x,y)=(x-e_{1},x-e_{2}),\]
where $e_{1}=0$ and $e_{2}=4u^2$.

\noindent
Let $H$ denote the group generated by points $P_{1},P_{2},P_{3}, T_{1}, T_{2}$ and let $G$ denote 
$E_{3}(K)$. The index $n$ equals $[G:H]$. There exist elements $R_{1},R_{2},R_{3}\in G$ such that $G$ is generated by $R_{1},R_{2},R_{3},T_{1},T_{2}$ and $H$ is generated by $a R_{1}, b R_{2},c R_{3},T_{1},T_{2}$, where $n=abc$ and $a|b|c$. For $n=8$, it follows that $(a,b,c)\in\{(1,1,8),(1,2,4),(2,2,2)\}$. For $n=4$, we have $(a,b,c)\in\{(1,1,4),(1,2,2)\}$ and for $n=2$ there is only one tuple $(a,b,c)=(1,1,2)$. Consider the modulo 2 map $\phi:G\rightarrow G/2G$ and $\eta=\psi\circ\phi$. The image $\eta(G)$ is isomorphic to $(\mathbb{Z}/2\mathbb{Z})^{4}$. If $n=8$, then $\eta(H)\cong (\mathbb{Z}/2\mathbb{Z})^{i}$, where $1\leq i\leq 3$. If $n=4$, then $\eta(H)\cong (\mathbb{Z}/2\mathbb{Z})^{i}$, where $2\leq i\leq 3$. If $n=2$, then $\eta(H)\cong (\mathbb{Z}/2\mathbb{Z})^3$. Hence, to show that $H=G$ it is sufficient to prove that $\eta(H)\cong (\mathbb{Z}/2\mathbb{Z})^4$. We easily compute
\begin{align*}
\eta(P_1)&=\left((t(t^2+5),(t^2+5)t \left(-5+\left(-2+2 \sqrt{2}\right) t-t^2\right) \left(-5+\left(2+2 \sqrt{2}\right) t-t^2\right)\right),\\
\eta(P_2)&=(1,t^4-4 t^3+6 t^2-20 t+25),\\
\eta(P_3)&=(\left(t^2-2 t+5\right) \left(t^2+2 t+5\right),1),\\
\eta(T_2)&=(t \left(t^2-2t+5\right) \left(t^2+2t+5\right)(t^{2}+5),t \left(t^4+4 t^3+6 t^2+20 t+25\right)(t^{2}+5))
\end{align*}
and prove that $|\eta(H)|=16$, which proves the theorem.
\end{proof}

\begin{corollary}\label{corollary:generators_over_Qt}
The group $E_{3}(\mathbb{Q}(t))$ is isomorphic to $\mathbb{Z}^{2}\oplus\mathbb{Z}/2\mathbb{Z}\oplus\mathbb{Z}/2\mathbb{Z}$. The free part  is generated by $P_{2}$, $P_{3}$. The torsion part is generated by $T_{1}$, $2T_{2}=(0,0)$.
\end{corollary}
\begin{proof}
First we prove that the rank of the group $E_{3}(\mathbb{Q}(t))$ equals $2$. From Corollary \ref{corollary:rank} we know that $\mathop{rank}E_{3}(\mathbb{Q}(t))\leq 3$. Since the points $P_{2}$ and $P_{3}$ are linearly independent it is enough to show that $\mathop{rank}E_{3}(\mathbb{Q}(t))$ is smaller than $3$.

Suppose to the contrary that the rank of $H=E_{3}(\mathbb{Q}(t))$ equals $3$. Then $H$ is a finite index subgroup of the group $G=E_{3}(\overline{\mathbb{Q}}(t))$ generated by $P_{1},P_{2},P_{3},T_{1}$ and $T_{2}$. Consider the $3$-dimensional $\mathbb{Q}$-vector space $G_{\mathbb{Q}}=G\otimes_{\mathbb{Z}}\mathbb{Q}$. There is a natural Galois representation 
\[\rho:\mathop{Gal}(\overline{\mathbb{Q}}/\mathbb{Q})\rightarrow\mathop{Aut}(G_{\mathbb{Q}}).\]
For $\sigma\in\mathop{Gal}(\overline{\mathbb{Q}}/\mathbb{Q})$ and $P\in G$ we define $\sigma(P\otimes 1)=\sigma(P)\otimes 1$, where $\sigma(P)$ denotes the element in $G$ such that $\sigma$ acts on the coefficients of rational functions in the coordinates of $P$. If $\sigma(\sqrt{-1})=-\sqrt{-1}$ and $\sigma(\sqrt{2})=\sqrt{2}$, then $\sigma(P_{1}\otimes 1)=-(P_{1}\otimes 1)$. In the basis $\{P_{1}\otimes 1,P_{2}\otimes 1,P_{3}\otimes 1\}$ of $G_{\mathbb{Q}}$, the matrix of the automorphism $\rho(\sigma)$ is 
\[\left(\begin{array}{ccc}-1 & 0 & 0\\ 0 & 1 & 0\\ 0 & 0 & 1\end{array}\right).\]
Hence, the representation $\rho$ is nontrivial. However, $G_{\mathbb{Q}}=H_{\mathbb{Q}}$, since we assumed that $H$ is of finite index in $G$. The representation $\rho$ acts trivially on $H_{\mathbb{Q}}$, which leads to a contradiction. Hence, $H$ is not of finite index in $G$, which implies that $\mathop{rank}H=2$.

If the Mordell-Weil lattice $E_{3}(\mathbb{Q}(t))/E_{3}(\mathbb{Q}(t))_{\mathop{tors}}$ were not generated by $P_{2}$ and $P_{3}$, then the lattice generated by those points would be of finite index greater than $1$ in the full Mordell-Weil lattice. Then the lattice generated by $P_{1}$, $P_{2}$ and $P_{3}$ would be of index greater than $1$ in the full Mordell-Weil lattice $E_{3}(\overline{\mathbb{Q}}(t))/E_{3}(\overline{\mathbb{Q}}(t))_{\mathop{tors}}$, which contradicts Lemma \ref{lemma:mordell_weil_generators}.

To conclude the proof, we compute the torsion part. Lemma \ref{lemma:torsion} shows that the torsion defined over $\mathbb{Q}(t)$ is generated by $T_{1}=(4u^2,0)$ and $(0,0)=2T_{2}$. It is the full torsion subgroup of $E_{3}(\mathbb{Q}(t))$.
\end{proof}

\begin{proof}[Proof of Theorem \ref{Ranks_two}]
From Corollary \ref{corollary:rank} it follows that $\mathop{rank}E_{3}(\overline{\mathbb{Q}}(t))=3$. Lemma \ref{lemma:mordell_weil_generators} and Lemma \ref{lemma:torsion} give explicit generators over $\overline{\mathbb{Q}}(t)$. Finally, Corollary \ref{corollary:generators_over_Qt} shows that the rank over $\mathbb{Q}(t)$ is $2$ and it gives explicit generators.
\end{proof}

\begin{proof}[Proof of Theorem \ref{Ranks_one}]
We apply the specialization theorem (cf. \cite[Theorem 11.4]{Silverman_book}) to the family
\begin{equation}\label{eq:formula_proof}
y^2=x(x-(u^2-1)^2)(x-4u^2)
\end{equation}
with a rational parameter $t$ and $u=\frac{2t}{5+t^2}$. The curve is nonsingular for any $t\neq 0$. Let $\frac{p}{q}=t$ denote a rational number where $p$ and $q\neq 0$ are integers. Let $\frac{P}{Q} = u = \frac{2pq}{p^2+5q^2}$ where $P$ and $Q\neq 0$ are integers. We claim that the triple $(a,b,c)=(P^2-Q^2,2PQ,P^2+Q^2)$ is an element of $\mathcal{S}$. Suppose to the contrary that $a=0$ or $b=0$. If $a=0$, then $P^2=Q^2$, hence $u=\pm 1$ and $(5+t^2)=\pm 2 t$, which does not have solutions $t\in\mathbb{Q}$, a contradiction. If $b=0$, then $P=0$ and $p=0$, so $t=0$, which is not possible, because the curve is nonsingular. The triple $(a,b,c)$ determines an associated elliptic curve
\[y^2=x(x-a^2)(x-b^2)\]
which has the following two points
\[Q_{1}=\left(\frac{1}{2} (a+b-c)^2,\frac{1}{2} (a+b) (a+b-c)^2\right),\]
\[Q_{2}=\left(\frac{1}{2} a (a-c),\frac{1}{2} a b \frac{1}{k^2}\left(p^4-25 q^4\right)\right),\]
where $k=\mathop{GCD}(2pq,p^2+5q^2)$.
The points $Q_{1}$ and $Q_{2}$ are obtained from 
\[P_{2}=(2(u - 1)^2,2(-1 + u)^2 (-1 + 2u + u^2))\]
and
\[P_{3}=\left(1 - u^2, \frac{\left(-5 + t^2\right) u \left(-1 + u^2\right)}{5 + t^2}\right)\]
by the map $(x,y)\mapsto \left(x\frac{(a-c)^2}{4},y\frac{(c-a)^3}{8}\right)$. The specialization theorem shows that the points $Q_{1}$ and $Q_{2}$ are linearly independent for almost all values of $t$. By Proposition \ref{proposition:bijection_classes} for all but finitely many elements of $\mathcal{S}/\sim$ the rank of the group of $\mathbb{Q}$-rational points on the curve $E_{(a,b,c)}$, $(a,b,c)\in \mathcal{S}/\sim$ is at least two. Hence for infinitely many $(a,b,c)\in\mathcal{S}$ the group $E_{(a,b,c)}(\mathbb{Q})$ has rank at least two.
\end{proof}

\begin{remark}
Observe that the point
\[(c^2,abc)=-2\left(\frac{1}{2} (a+b-c)^2,\frac{1}{2} (a+b) (a+b-c)^2\right)\]
is on the curve
\[y^2=x(x-a^2)(x-b^2).\]
The point $\left(\frac{1}{2} (a+b-c)^2,\frac{1}{2} (a+b) (a+b-c)^2\right)$ corresponds to the point
\[\left(2(t-1)^2,2 (t-1)^2 (-1+2t+t^2)\right)\]
via the inverse of the map $(x,y)\mapsto \left(x\frac{(a-c)^2}{4},y\frac{(c-a)^3}{8}\right)$. The point 
\[\left(2(t-1)^2,2 (t-1)^2 (-1+2t+t^2)\right)\]
is a generator of the free part of the Mordell-Weil group over $\mathbb{Q}(t)$ on the curve
\[y^2=x(x-(t^2-1)^2)(x-4t^2).\]
We prove this fact in the next lemma.
\end{remark}

\begin{lemma}\label{lemma:rank_K3}
The group $E_{2}(\overline{\mathbb{Q}}(t))$ is isomorphic to $\mathbb{Z}^{2}\oplus\mathbb{Z}/2\mathbb{Z}\oplus\mathbb{Z}/4\mathbb{Z}$. The free part is generated by points.
\begin{align*}
P_{1}&=(2 (1 + \sqrt{2}) (-1 + t)^2 t,2\sqrt{-1} (1 + \sqrt{2}) (-1 + (\sqrt{2} - t)^2) (-1 + t)^2 t),\\
P_{2}&=(2(t - 1)^2,2(-1 + t)^2 (-1 + 2t + t^2)).
\end{align*}
The torsion part is generated by 
\begin{align*}
T_{1}&=(-4t^2,0)\\
T_{2}&=(2(-t + t^3),2\sqrt{-1}(t^2-1)t(-1 - 2t + t^2)).
\end{align*}
The group $E_{2}(\mathbb{Q}(t))$ is generated by $P_{2}$, $T_{1}$ and $2T_{2}=(0,0)$. 
\end{lemma}
\begin{proof}
The torsion subgroup is computed similarly as in the proof of Lemma \ref{lemma:torsion}. We put $K=\overline{\mathbb{Q}}(t)$. Let $(E_{2}(K)/E_{2}(K)_{\mathop{tors}},\langle\cdot,\cdot\rangle_{E_{2}})$ be the Mordell-Weil lattice with the height pairing $\langle\cdot,\cdot\rangle_{E_{2}}$. We compute easily $\langle P_{1},P_{1}\rangle_{E_{2}}=\frac{1}{2}$, $\langle P_{2},P_{2}\rangle_{E_{2}} = 1$ and $\langle P_{1},P_{2}\rangle_{E_{2}} = 0$. In general, for each $P,Q\in E_{2}(K)/E_{2}(K)_{\mathop{tors}}$ the value of the pairing $\langle P,Q\rangle_{E_{2}}$ lies in $\frac{1}{4}\mathbb{Z}$ which follows by the type of singular fibers (cf. Table \ref{table:singular_E2}). 

Consider the lattice $\Lambda=(E_{2}(K)/E_{2}(K)_{\mathop{tors}}$ with the pairing $\langle\cdot,\cdot,\rangle=4\langle\cdot,\cdot\rangle_{E_{2}}$. Let $\Lambda '$ be generated by $P_{1}$ and $P_{2}$. It is a sublattice of $\Lambda$ of a finite index which we call $n=[\Lambda:\Lambda ']$. For the lattice $\Lambda$ we have $\langle P_{1},P_{1}\rangle=2$, $\langle P_{2},P_{2}\rangle = 4$ and $\langle P_{1},P_{2}\rangle = 0$. Hence, the following equality holds for discriminants of lattices $\Lambda$ and $\Lambda '$ with respect to the pairing $\langle\cdot,\cdot\rangle$
\[8=\Delta(\Lambda ') = n^2\Delta(\Lambda).\]
Hence, $n$ divides $2$. We want to show that $n=1$. Suppose to the contrary that $n=2$. There exists a point $R\in E_{2}(K)$ of infinite order, such that 
\[2R=aP_{1}+bP_{2}+T\]
for some $a,b\in\{0,1\}$ and $T\in E_{2}(K)_{\textrm{tors}}$. So
\[4\langle R,R\rangle=\langle 2R,2R\rangle=2a^2+4b^2=2(a^2+2b^2).\]
This implies $2\mid (a^2+2b^2)$. For $a,b\in\{0,1\}$ there are pairs $(a,b)=(0,0)$ and $(a,b)=(0,1)$. For $(a,b)=(0,0)$ we obtain the equation 
\[2R=T\]
for a $K$-rational torsion point $T$. This implies that $R$ is of finite order, hence a contradiction. For the pair $(a,b)=(0,1)$ we obtain the equation 
\[2R=P_{2}+T\]
with a $K$-rational torsion point $T$. We consider only the cases $T=O$, $T=T_{1}$, $T=T_{2}$ and $T=T_{1}+T_{2}$, since one can add a point from $2 E_{2}(K)_{\textrm{tors}}$ to both sides. 

Consider the $2$-descent homomorphism
\[\psi:E_{2}(K)/2E_{2}(K)\hookrightarrow K^{*}/(K^{*})^{2}\times K^{*}/(K^{*})^{2}.\]
The homomorphism is defined for non-torsion points $(x,y)$ in $E_{2}(K)$ by the formula $\psi(x,y)=(x-e_{1},x-e_{2})$ where $e_{1}=0$ and $e_{2}=4t^2$. We check using MAGMA that $\psi(P_{2}+T)\neq (1,1)$ for $T\in\{O,T_{1},T_{2},T_{1}+T_{2}\}$. This proves that the assumption $n=2$ leads to a contradiction. Hence $\Lambda=\Lambda '$, proving that the rank of $E_{2}(K)$ is two.

Now we prove that the group $E_{2}(\mathbb{Q}(t))$ is generated by $P_{2}$, $T_{1}$ and $2T_{2}$. For the torsion part, observe that $E_{2}(\mathbb{Q}(t))_{\textrm{tors}}\subset E_{2}(\overline{\mathbb{Q}}(t))_{\textrm{tors}}$. The group $E_{2}(\overline{\mathbb{Q}}(t))_{\textrm{tors}}$ is generated by $T_{1}$ and $T_{2}$. Since $T_{2}$ is not $\mathbb{Q}(t)$-rational, the group $E_{2}(\mathbb{Q}(t))_{\textrm{tors}}$ is generated by $T_{1}$ and $2T_{2}=(0,0)$. We know that the rank of $E_{2}(\overline{\mathbb{Q}}(t))$ is $2$. Hence, the rank of $E_{2}(\mathbb{Q}(t))$ is at most $2$. Assume it equals 2.

Then there exists a point $R$ defined over $\mathbb{Q}(t)$ such that $R=aP_{1}+bP_{2}+T$ for some integers $a\neq 0$ and $b$ and a torsion point $T$. Since $4T=O$, we have
\begin{equation}\label{equation:combination}
4R=4aP_{1}+4bP_{2}.
\end{equation}
Recall that
\[P_{1}=(2 (1 + \sqrt{2}) (-1 + t)^2 t,2\sqrt{-1} (1 + \sqrt{2}) (-1 + (\sqrt{2} - t)^2) (-1 + t)^2 t).\]
We choose an automorphism $\sigma\in\mathop{Gal}(\overline{\mathbb{Q}}/\mathbb{Q})$ which acts on the coefficients of rational functions in the coordinates of $P_{1}$ by the formula $\sigma(\sqrt{-1})=-\sqrt{-1},\, \sigma(\sqrt{2})=\sqrt{2}$. The action of $\sigma$ commutes with the addition morphism on the curve $E_{2}$ which is defined over $\mathbb{Q}(t)$.
Applying $\sigma$ to both sides of (\ref{equation:combination}) we get $8aP_{1}=O$, because $\sigma(P_{1})=-P_{1}$ and $\sigma(P_{2})=P_{2}$. This gives a contradiction since $P_{1}$ is a non-torsion point. 

If the Mordell-Weil lattice $E_{2}(\mathbb{Q}(t))/E_{2}(\mathbb{Q}(t))_{\mathop{tors}}$ were not generated by $P_{2}$, then the lattice generated by this point would be of finite index greater than $1$ in the full Mordell-Weil lattice. Then the lattice generated by $P_{1}$ and $P_{2}$ would be of index greater than $1$ in the full Mordell-Weil lattice $E_{2}(\overline{\mathbb{Q}}(t))/E_{2}(\overline{\mathbb{Q}}(t))_{\mathop{tors}}$, which contradicts what has been proven already.
\end{proof}

\begin{proof}[Proof of Theorem \ref{theorem:rank_K3}]
This follows from Lemma \ref{lemma:rank_K3} and the fact that the curves
\[y^2=x(x-1)(x-\left(\frac{2t}{t^2-1}\right)^2),\]
\[y^2=x(x-(t^2-1)^2)(x-4t^2)\]
are isomorphic over $\mathbb{Q}(t)$. 
\end{proof}
\begin{remark}
It is natural to ask what is the rank of the Mordell-Weil group of a curve
\[y^2=x(x-\alpha a^2)(x-\beta b^2),\]
where $\alpha a^2+\beta b^2+\gamma c^2=0$ for some $\alpha,\beta,\gamma \in\mathbb{Z}$. 
In particular, one would like to know what is the upper bound of the rank in such a big family. We hope to return to this question in the future.
\end{remark}

\section*{Acknowledgments}
The author would like to thank Wojciech Gajda for suggesting this research problem and for many helpful remarks. He thanks Adrian Langer for many helpful comments. The author gratefully acknowledges the many helpful suggestions of Remke Kloosterman, especially drawing the author's attention to the method of twisting and the application of the Artin-Tate conjectures. Finally, the author wishes to express his thanks to an anonymous referee for several remarks that improved the exposition and removed several inaccuracies. The author was supported by the National Science Centre research grant 2012/05/N/ST1/02871.

\bibliography{bibliography}

\providecommand{\bysame}{\leavevmode\hbox to3em{\hrulefill}\thinspace}
\providecommand{\MR}{\relax\ifhmode\unskip\space\fi MR }
\providecommand{\MRhref}[2]{%
  \href{http://www.ams.org/mathscinet-getitem?mr=#1}{#2}
}
\providecommand{\href}[2]{#2}
\begin{thebibliography}{10}

\bibitem{Artin_SD_K3_Tate_Conjectures}
M.~Artin and H.~P.~F. Swinnerton-Dyer, \emph{The {S}hafarevich-{T}ate
  conjecture for pencils of elliptic curves on {$K3$} surfaces}, Invent. Math.
  \textbf{20} (1973), 249--266.

\bibitem{Hulek}
Wolf~P. Barth, Klaus Hulek, Chris A.~M. Peters, and Antonius Van~de Ven,
  \emph{Compact complex surfaces}, second ed., Ergebnisse der Mathematik und
  ihrer Grenzgebiete. 3. Folge. A Series of Modern Surveys in Mathematics
  [Results in Mathematics and Related Areas. 3rd Series. A Series of Modern
  Surveys in Mathematics], vol.~4, Springer-Verlag, Berlin, 2004.

\bibitem{Dolgachev}
Fran{\c{c}}ois~R. Cossec and Igor~V. Dolgachev, \emph{Enriques surfaces. {I}},
  Progress in Mathematics, vol.~76, Birkh\"auser Boston Inc., Boston, MA, 1989.

\bibitem{Elsenhans_Jahnel}
Andreas-Stephan Elsenhans and J{\"o}rg Jahnel, \emph{The picard group of a k3
  surface and its reduction modulo p}, Algebra Number Theory \textbf{5} (2011),
  no.~8, 1027--–1040.

\bibitem{INK_Pythagorean}
F.~A. {Izadi}, K.~{Nabardi}, and F.~{Khoshnam}, \emph{{On a Family Of Elliptic
  Curves With Positive Rank arising from Pythagorean Triples}}, ArXiv e-prints
  (2010), arXiv:1012.5837v4.

\bibitem{Kleiman_in_dix_expose}
S.~L. Kleiman, \emph{Algebraic cycles and the {W}eil conjectures}, Dix
  espos\'es sur la cohomologie des sch\'emas, North-Holland, Amsterdam, 1968,
  pp.~359--386.

\bibitem{Kloosterman_rank_15}
Remke Kloosterman, \emph{Elliptic {$K3$} surfaces with geometric
  {M}ordell-{W}eil rank 15}, Canad. Math. Bull. \textbf{50} (2007), no.~2,
  215--226.

\bibitem{Milne_Tate_Conjecture}
James~S. Milne, \emph{On a conjecture of {A}rtin and {T}ate}, Ann. of Math. (2)
  \textbf{102} (1975), no.~3, 517--533.

\bibitem{Milne_etale_book}
\bysame, \emph{\'{E}tale cohomology}, Princeton Mathematical Series, vol.~33,
  Princeton University Press, Princeton, N.J., 1980.

\bibitem{Milne_Tate_Conjecture_2}
\bysame, \emph{Values of zeta functions of varieties over finite fields}, Amer.
  J. Math. \textbf{108} (1986), no.~2, 297--360.

\bibitem{Mumford_Selected_papers}
David Mumford, \emph{Selected papers on the classification of varieties and
  moduli spaces}, Springer-Verlag, New York, 2004, With commentaries by David
  Gieseker, George Kempf, Herbert Lange and Eckart Viehweg.

\bibitem{Schutt_Picard}
Matthias Sch{\"u}tt, \emph{{$K3$} surfaces with {P}icard rank 20}, Algebra
  Number Theory \textbf{4} (2010), no.~3, 335--356.

\bibitem{Shioda_Schutt}
T.~Shioda and M.~Sch{\"u}tt, \emph{Elliptic surfaces}, ArXiv e-prints (2010),
  arXiv:0907.0298v3.

\bibitem{Shioda_Mordell_Weil}
Tetsuji Shioda, \emph{On the {M}ordell-{W}eil lattices}, Comment. Math. Univ.
  St. Paul. \textbf{39} (1990), no.~2, 211--240.

\bibitem{Silverman_arithmetic}
Joseph~H. Silverman, \emph{The arithmetic of elliptic curves}, Graduate Texts
  in Mathematics, vol. 106, Springer-Verlag, New York, 1986.

\bibitem{Silverman_book}
\bysame, \emph{Advanced topics in the arithmetic of elliptic curves}, Graduate
  Texts in Mathematics, vol. 151, Springer-Verlag, New York, 1994.

\bibitem{Top_Zeeuw}
Jaap Top and Frank De~Zeeuw, \emph{Explicit elliptic {$K3$} surfaces with rank
  15}, Rocky Mountain J. Math. \textbf{39} (2009), no.~5, 1689--1697.

\bibitem{Luijk_Matrices}
Ronald van Luijk, \emph{A {$K3$} surface associated with certain integral
  matrices having integral eigenvalues}, Canad. Math. Bull. \textbf{49} (2006),
  no.~4, 560--577.

\bibitem{Luijk_Heron}
\bysame, \emph{An elliptic {$K3$} surface associated to {H}eron triangles}, J.
  Number Theory \textbf{123} (2007), no.~1, 92--119.

\end{thebibliography}
\bibliographystyle{amsplain}

\end{document}